%% file: main.tex
\documentclass[11pt]{article}
\usepackage{epsfig}
\usepackage{amsmath}
\usepackage{amssymb}
\usepackage{boxedminipage}
\usepackage{algorithm}

\usepackage{tikz}
\usepackage{tkz-berge}
\usepackage{tkz-graph}

\newcommand {\be}[1]{\begin{equation}\label{#1}}
\newcommand {\ee}{\end{equation}}
\newcommand {\bea}{\begin{eqnarray}}
\newcommand {\eea}{\end{eqnarray}}

\newcommand{\qed}{\hfill \rule{7pt}{7pt}}

\newtheorem{theorem}{Theorem}
\newtheorem{lemma}{Lemma}

\newtheorem{prop}{Proposition}
\newtheorem{coro}{Corollary}

\begin{document}

\title{Strategyproof and Consistent Rules for Bipartite Flow Problems}

	\author{
Shyam Chandramouli\thanks{IEOR Department, Columbia University, New York,
NY; {\tt sc3102@columbia.edu}} and
Jay\
Sethuraman\thanks{IEOR Department, Columbia University, New York,
NY;
{\tt jay@ieor.columbia.edu}}}


\maketitle

\begin{abstract}
We continue the study of Bochet et al.~\cite{bim,bims} and Moulin and Sethuraman~\cite{ms,mslc} on fair
allocation in bipartite networks. In these models, there is a moneyless market, in which
a non-storable, homogeneous commodity is reallocated between agents with single-peaked
preferences. Agents are either suppliers or demanders. While the egalitarian rule of Bochet et al.~\cite{bim,bims}  satisfies \emph{pareto optimality, no envy and strategyproof} , it is not consistent. On the other hand, the work of Moulin and Sethuraman~\cite{ms,mslc} is related to \emph{consistent} allocations and rules that are extensions of the uniform rule. We bridge the two streams of work by introducing the \emph{edge fair} mechanism which is both \emph{consistent} and \emph{groupstrategyproof}. On the way, we explore the ``price of consistency" i.e. how the notion of consistency is fundamentally incompatible with certain notions of fairness like Lorenz Dominance and No-Envy. The current work also introduces the idea of strong invariance as desideratum for \emph{groupstrategyproofness} and generalizes the proof of Chandramouli and Sethuraman~\cite{cs} to a more broader class of mechanisms. 
Finally, we conclude with the study of the edge fair mechanism in a transshipment model where the strategic agents are on the links connecting different supply/demand locations. 
\end{abstract}

\newpage

\input{introduction.tex}
\input{maxflow.tex}

\input{model1.tex}
\input{model2.tex}

\newpage
\nocite{*}
\bibliographystyle{plain}
\bibliography{refs}

\end{document}

%% file: introduction.tex
\section{Introduction}
\label{intro}
We study the problem of fair division of a maximum flow in a capacitated bipartite network. This model
generalizes a number of matching and allocation problems that have been studied extensively over
the years, motivated by applications in school choice, kidney exchange etc. The common feature in these
application contexts is that the associated market  is moneyless, so that fairness is achieved by
equalizing the allocation {\em as much as possible}. This last caveat is to account for additional
considerations, such as Pareto efficiency and strategyproofness, that may be part of the planner's objective.

Specifically, we are given a bipartite network $G = (S \cup D, E)$, and we think of $S$ as the set of supply
nodes and $D$ as the set of demand nodes. Each arc $(i,j) \in E$ connects a supply node $i$ to a demand
node $j$, and has capacity $u_{ij} \geq 0$. There is a single commodity (the resource) that is available at
the supply nodes and needs to be transferred to the demand nodes: we assume that supply node $i$ has
$s_i$ units of the resource, and demand node $j$ requires $d_j$ units of it. The capacity of an arc $(i,j)$
is interpreted as an upper bound on the direct transfer from supply node $i$ to demand node $j$. The
goal is to satisfy the demands ``as much as possible'' using the available supplies, while also respecting
the capacity constraints on the arcs.

A well-studied special case of our problem is that of allocating a single resource (or allocating the
resource available at a single location) amongst a set of agents with varying (objectively verifiable) claims on it. 
This is the special case when there is a single supply node that is connected to every one of the demand
nodes in the network by an arc of large-enough capacity. If the
sum of the claims of the agents exceeds the amount of the resource available, the problem
is a standard rationing problem (studied in the literature as ``bankruptcy'' problems or ``claims''
problems). There is an extensive literature devoted to such problems that has resulted in a
thorough understanding of many natural methods including the {\em proportional} method,
the {\em uniform gains} method, and the {\em uniform losses} method. A different view of
this special case is that of allocating a single resource amongst agents with single-peaked
preferences over their net consumption. Under this view, studied by Sprumont~\cite{spr},
Thomson~\cite{thomson} and many others, the goal is to design a mechanism for allocating
the resource that satisfies appealing efficiency and equity properties, while also eliciting
the preferences of the agents truthfully. The {\em uniform rule}, which is essentially an adaptation 
of the uniform gains method applied to the reported peaks of the agents, occupies a central
position in this literature: it is strategy-proof (in fact, group strategy-proof), and finds an
envy-free allocation that Lorenz dominates every other efficient allocation; furthermore,
this rule is also {\em consistent}. (We will define consistency, Lorenz dominance, etc. precisely
in Section~\ref{s:maxflow}.) A natural two-sided version of Sprumont's model has agents initially 
endowed with some amount of the resource, so that agents now fall into two categories: someone
endowed with less than her peak is a potential demander, whereas someone endowed with more than
her peak is a potential supplier. The simultaneous presence
of demanders and suppliers creates an opportunity to trade, and the obvious adaptation of
the uniform rule gives their peak consumption to agents on the short
side of the market, while those on the long side are uniformly rationed (see 
\cite{KPS}, \cite{BJ}). This is again equivalent to a standard rationing problem because
the nodes on the short side of the market can be collapsed to a single node.
The model we consider generalizes this by assuming that the resource can only be transferred 
between certain pairs of agents. Such constraints are typically logistical (which supplier can
reach which demander in an emergency situation, which worker can
handle which job request), but could be subjective as well (as when a
hospital chooses to refuse a new patient by declaring red
status). This complicates
the analysis of efficient (Pareto optimal) allocations, because short demand
and short supply typically coexist in the same market. 

The general model we consider in this paper has been the subject of much recent
research and was first formulated by Bochet et al.~\cite{bim, bims}. The authors
work with a bipartite network in both papers and assume that each node is populated
by an agent with single-peaked preferences over his consumption of the resource: thus,
each supply node has an ``ideal'' supply (its peak) quantity, and each demand node has
an ideal demand. These preferences are assumed to be private information, and Bochet
et al.~\cite{bim, bims} propose a clearinghouse mechanism that collects from each agent
only their ``peaks'' and picks Pareto-optimal transfers with respect to the reported peaks.
Further, they show that their mechanism is strategy-proof in the sense that it is a dominant
strategy for each agent to report their peaks truthfully. While the models in the two papers
are very similar, there is also a critical difference: in~\cite{bims}, the authors require that
no agent be allowed to send or receive any more than their peaks, whereas in~\cite{bim} the
authors assume that the {\em demands} must be satisfied exactly (and so some supply nodes
will have to send more than their peak amounts). The mechanism of Bochet et al.---the {\em egalitarian}
mechanism---generalizes
the uniform rule, and finds an allocation that Lorenz dominates all Pareto efficient allocations.
Later, Chandramouli and Sethuraman~\cite{cs} show that the egalitarian mechanism is in fact
group strategyproof: it is a dominant strategy for any group of agents (suppliers or demanders) 
to report their peaks truthfully \footnote{Szwagrzak~\cite{sz1} genrealize the proof methodology of Chandramouli and Sethuraman~\cite{cs} to establish that all separably convex rules are group strategyproof }. 
Szwagrzak~\cite{sz1,sz2,sz3} expands the study of allocation rules in these 
networked economies
by introducing broader class of mechanisms with various fairness properties. His work also 
develops alternative characterizations of these mechanisms (in p
articular, the 
egalitarian mechanism) and provides a unified view of the allocation problem on networks.
Szwagrzak~\cite{sz1} studies the property of contraction invariance of an allocation rule: when 
the set of feasible allocations contracts such that the optimal allocation is still 
in this smaller set, then the allocation rule should continue to select the same allocation.
He shows that the egalitarian rule is contraction invariant. 
These results suggest that the egalitarian mechanism may be the
correct generalization of the uniform rule to the network setting. However, it is fairly easy to show
that the egalitarian mechanism is not consistent: if the link from a supply node $i$ and demand node $j$
is dropped, and $s_i$ and $d_j$ are adjusted accordingly, applying the egalitarian mechanism to the
reduced problem will not necessarily give the same allocation to the agents. Motivated by this
observation, Moulin and Sethuraman~\cite{ms, mslc} study rules for network rationing
problems that extend a given rule for a standard rationing problem while preserving consistency
and other natural axioms. In particular, they propose a family of rules that generalize the
uniform rule to the bipartite network setting. While they are able to show that their extension
satisfies consistency, it is not known if any of these rules is strategyproof.

Our main contribution in this paper is a new {\em group strategy-proof} mechanism (the ``edge-fair'' mechanism)
that is a consistent extension
of the uniform rule. Our proof shows that for any Pareto efficient mechanism, group strategyproofness is equivalent to a property that we call {\em strong invariance} that is often straightforward to verify. (In particular, the group strategy-proofness of the egalitarian mechanism that we established in an earlier paper also follows immediately, even if one
works with a capacitated model.)
Along the way we show that consistency imposes very severe restrictions: for instance,
no consistent rule can find allocations that are envy-free, even in the limited sense introduced by Bochet
et al.~\cite{bims} for such problems. The mechanism we propose does not find the Lorenz
optimal allocation, but we show that no consistent mechanism can. 
We consider a related model where the supplies and demands at the nodes are given, but that each edge is
controlled by an independent agent with single-peaked preferences on the amount transferred along that edge.
The planner still wishes to implement a maximum flow (it is now a design constraint), and the goal is to divide this
reasonably among the edges of the network.
For this model we show that a Lorenz optimal allocation need not exist, but that our mechanism can still
be applied and finds a reasonable division of the max-flow.

The rest of the paper is organized as follows: in Section~\ref{s:maxflow} we consider the standard model
of maximizing the total flow in a capacitated bipartite network.  We state the well-known Gallai-Edmonds
decomposition, and describe the edge-fair algorithm that selects a particular max-flow for any given
problem. An easy argument shows that the edge-fair algorithm makes a consistent selection of max-flows
across related problems. Section~\ref{s:model1} considers the model in which agents are located on the nodes of the
network and have single-peaked preferences over their allocations---the equivalence of group strategy-proofness
and strong invariance, and the fact that the edge-fair rule satisfies strong invariance are the key results in
this section. In section~\ref{s:model2} we turn to the problem in which agents are on the edges of the network,
and study the implications of consistency. 

%% file: maxflow.tex
\section{Maximum Flows and the Edge-Fair Algorithm}
\label{s:maxflow}

\subsection{Model}
We consider the problem of transferring a single commodity
from a set $S$ of suppliers to a set $D$ of demanders using
a set $E$ of edges.  Supplier $i$ has $s_i \geq 0$ units of
the commodity, and demander $j$ wishes to consume $d_j \geq 0$.
Associated with each edge is a distinct
supplier-demander pair: the edge $e = (i,j)$ links supplier
$i$ to demander $j$, and has a non-negative, possibly infinite,
capacity $u_{ij}$. Transfer of the commodity is allowed between
supplier $i$ and demander $j$ only if $(i,j) \in E$, in which case
at most $u_{ij}$ units of the commodity can be transferred along
this edge\footnote{Equivalently, we could assume that an edge
exists between every supplier $i$ and every demander $j$,
but that $u_{ij} = 0$ for all $(i,j) \not \in E$.}.
The goal is to find an ``optimal'' transfer of the commodity from the
suppliers to the demanders.
We let $G = (S \cup D, E)$ be the natural bipartite graph and we speak
of the {\em problem} $(G,s,d,u)$.

We use the following notation. For any subset $T\subseteq S$,
the set of demanders compatible with the suppliers in $T$ is $%
f(T)=\{j\in D| (i,j) \in E, \;\; i \in T \; \}$. 
Similarly, the set of suppliers compatible with the demanders 
in $C \subseteq D$ is $g(C)=\{i\in S| (i,j) \in E, \;\; j \in C \}$. 
We abuse notation and say $f(i)$ and $g(j)$ instead of
$f(\{i\})$ and $g(\{j\})$ respectively.
For any subsets $T \subseteq S$, $C \subseteq D$, $x_T :=
\sum_{i \in T} x_i$ and $y_C := \sum_{j \in C} y_j$.

A transfer of the commodity from $S$ to $D$ is realized by a flow $\varphi$,
which specifies the amount of
the commodity transferred from supplier $i$ to demander $j$ using
the edge $(i,j) \in E$. 
The flow $\varphi$ induces an allocation
vector for each supplier and each demander as follows:
\begin{equation}
\text{for all }i\in S:x_{i}(\varphi )=\sum_{j\in f(i)}\varphi _{ij};\text{
for all }j\in D:y_{j}(\varphi )=\sum_{i\in g(j)}\varphi _{ij}  \label{1}
\end{equation}%
The flow $\varphi$ is {\em feasible} if (i) $\varphi_{ij} \leq u_{ij}$ for all $(i,j) \in E$ and $\varphi_{ij} = 0$ for all
$(i,j) \not \in E$;
 (ii) $x_i(\varphi) \leq s_i$ for all $i \in S$; and
(iii) $y_j(\varphi) \leq d_j$ for all $j \in D$.
Let ${\cal F}(G,s,d,u)$ be the set of feasible flows for the problem $(G,s,d,u)$.
A feasible flow $\varphi^*$ is a maximum flow if 
$$\varphi^* \in \arg \max_{\varphi \in {\cal F}(G,s,d,u)} \sum_{i \in S} x_i(\varphi).$$
Let
${\cal F}^*(G,s,d,u)$ be the set of maximum flows for the problem $(G,s,d,u)$.
For reasons that will be clearer later, we shall focus mostly on finding a maximum flow for
any given problem. As a result, it is important to understand the set ${\cal F}^*(G,s,d,u)$,
which we turn to next.

\paragraph{The Gallai-Edmonds Decomposition.}
The problem under consideration is the well-known problem of finding a maximum flow
in a capacitated bipartite network. The following result characterizes the structure
of maximum flows and is essentially a version of the Gallai-Edmonds decomposition.
It can proved by a straightforward application of the max-flow min-cut theorem.

\begin{lemma}
\label{lem1}
There exists a partition $S_{+} , S_{-}$ of $S$, and a partition $D_{+}, D_{-}$ of $D$
such that
the flow $\varphi$ with net transfers $x,y$
is a maximum flow  if and only if 
\begin{equation}
 \varphi_{ij} = u_{ij} \hspace{2mm} \forall ij \in G(S_{-},D_{-}), \hspace{2mm} x_{i} = s_{i} \hspace{2mm} \forall i \in S_{+}, \hspace{2mm} y_{j}=d_{j} \hspace{2mm} \forall j \in D_{+} \\ 
\end{equation}
\end{lemma}

\begin{proof}
Let $\lambda := (\lambda_i)_{i \in S}$ be non-negative.
Construct the following network $G(\lambda)$: introduce a source $s$ and
a sink $t$; arcs of the form $(s,i)$ for each supplier $i$ with capacity $\lambda_i$,
arcs of the form $(j,t)$ for each demander $j$ with capacity $d_j$; an arc of capacity $u_{ij}$
from supplier $i$ to demander $j$ if supplier $i$ and demander $j$ share a link.
Consider now a maximum $s$-$t$ flow $\varphi$ in the network $G(s)$.
By the max-flow min-cut theorem, there is a cut $C$ (a cut is a subset 
of nodes that contains the source $s$ but not the sink $t$) whose capacity 
is equal to that of the max-flow. Let $X$ be the set of suppliers in $C$ and $Y$
be the set of demanders in $C$.
If the min-cut is not unique, it is again well-known 
(see~\cite{lp88}) that there is a min-cut
with the largest $X$ (largest in the sense of inclusion), and a min-cut with the
smallest $X$ (again in the sense of inclusion). Call these 
sets $\overline{X}$ and $\underbar{X}$.
Define
$S_{-} := \underbar{X}$ and $S_{+} := S \setminus S_{-}$; and define
$D_{+} = f(S_{-}) \cap \{j |y_{j} (\varphi) = d_{j}, \;\; \forall \varphi \in {\cal F}^*\}$ and $D_{-} = D \setminus D_{+}$.
We note that the partition is uniquely determined for each problem.\footnote{It is easy to check that every 
supplier in $\overline{X} \setminus \underbar{X}$ will transfer his entire
supply in {\em all} maximum flows.}

Let $C$ be the cut that has precisely $\underbar{X}$ as its set of suppliers and $Y$ as its set of demanders. We claim that
$Y \subseteq f(\underbar{X})$. For otherwise, there is a 
supplier $j \in Y \setminus f(X)$ who contributes $d_j$ to the capacity
of the cut $C$, and omitting $j$ from $C$ would reduce this capacity by $d_j > 0$, resulting in a smaller capacity cut.
Moreover, by the max-flow min-cut theorem, every edge $(i,j)$ with $i \in C$ and $j \not \in C$ must carry flow
equal to its capacity, and that the value of the max-flow is precisely the sum of the capacites of such edges. Thus,
$\varphi_{ij} = u_{ij}$ for
every edge $(i,j)$ with $i \in S_{-}$ and $j \in D_{-}$; the edge $(s,i)$ carries a flow of $s_i$ for each
supplier $i \in S_{+}$; and the edge $(j,t)$ carries a flow of $d_j$ for each demander $j \in D_{+}$.
The lemma now follows.
\end{proof}
\qed

\subsection{Axioms}

We briefly describe some key axioms that we want our rules to satisfy.

\paragraph{Edge consistency.}
The key axiom in our paper is {\em edge consistency} (or simply consistency, hereafter). Suppose we have
a rule $\varphi$ that picks a flow $z$ for a given problem $(G,s,d,u)$.
Fix an edge $(i,j) \in G$ connecting supply node $i$ and demand node $j$,
and define the reduced problem as follows: the new graph is $G' = G \setminus \{e\}$; the
supplies and demands of all the nodes other than $i$ and $j$ stay the same, $s'_{i} = s_i - z_{ij}$ and
$d'_{j} = d_j - z_{ij}$; and the capacities on the edges that remain stay the same. Let $z'$ be the
flow picked by the rule $\varphi$ for the reduced problem $(G',s',d', u)$. The rule $\varphi$ is
edge-consistent if $z = z'$ for every reduced problem $G'$ that can be obtained from $G$ by omitting
an edge. Observe that
$z$ restricted to the edges in $G'$ is a max-flow for the reduced problem, and edge-consistency requires
that the rule not ``reallocate'' the flow amongst the remaining edges if some edge is dropped from the
problem and the corresponding supplies and demands are adjusted in the obvious way. 

\paragraph{Continuity.}
The mapping $\varphi:(G,s,d,u) \rightarrow \mathbb{R}^{|E|}$ is jointly continuous in $s$, $d$, and $u$.
Roughly speaking, this simply says that a rule is continuous only if 
small changes in supplies, demands or edge-capacities result in small changes on the edge-flows picked
by the rule.

\paragraph{Symmetry.}
Consider any permutation $\pi$ of the supply nodes and a permutation $\sigma$ of the demand nodes. Define
the graph $G'$ as follows: $(i,j) \in G$ if and only if $(\pi(i), \sigma(j)) \in G'$. The
supplies $s'$ and demands $d'$ are defined analogously by permuting the original supplies
and demands according to the respective permutations. Likewise for the capacities.
A rule $\varphi$ is symmetric if and only if for every
$\pi$ and every $\sigma$, $z_{ij} = z'_{\pi(j),\sigma(j)}$ where $z$ and $z'$ are the outcomes of
the rule for the problems $(G,s,d,u)$ and $(G',s',d',u')$ respectively.

\subsection{The Edge-Fair Algorithm}
\label{sec:ef_rule}
Given two max-flows $\varphi$ and $\varphi'$ sorted in increasing order we say that $\varphi$ \emph{lexicographically dominates} $\varphi'$ if the first coordinate 
$k$ in which $\varphi$ and $\varphi'$ are not equal is such that 
$\varphi_{k} > \varphi'_{k}$. 
(Note that the $k$-th smallest entry in the flows $\varphi$ and $\varphi'$ may
be on different edges.)
The max-flow $\varphi$ is lex-optimal if it lexicographically dominates all other max-flows ${\cal F}^{*}(G,s,d,u)$. It is
clear that a lex-optimal flow exists and is unique.\footnote{The term lex-optimal flow is also used to mean a flow
whose induced allocation for the suppliers (or demanders) lex-dominates the induced allocation for the suppliers
in any other flow~\cite{meg74, meg77}.}
The edge-fair algorithm, formally described next,
finds this lex-optimal flow by solving a sequence of linear programming problems.

We fix a problem $(G,s,d)$ such that $s_{i}, d_{j} > 0 $ for all $i,j$ (clearly, if $s_{i} =0$ or $d_{j} = 0 $ we can 
ignore supplier $i$ or demander $j$ altogether). Let $E_0 := E$ and ${\cal F}_0 := {\cal F}^*(G,s,d,u)$, the set
of all max-flows for the given problem.
The edge-fair algorithm (or rule) proceeds iteratively, solving a linear programming 
problem in each step. In any iteration $t$, it starts with a candidate set
of max-flows ${\cal F}_t$, and a set of active edges $E_t$, and solves
the following linear programming problem:
$$
\max_{ \varphi \in {\cal F}_t} \; \bigg\{ \lambda_{t+1} \;  \bigg|  \; \varphi_{e} \geq \lambda_{t+1}, \;\; \forall e \in E_t \bigg\}.
$$
Suppose $\lambda^*_{t+1}$ is the optimal value of this linear programming problem. Then,
$${\cal F}_{t+1} \; = \; \bigg\{ \varphi \in {\cal F}_t \; | \; \varphi_{e} \geq \lambda^*_{t+1} \;\; \forall e \in E_t \bigg \},$$
and
$$E_{t+1} \; = \;  \bigg\{ e \in E_{t} \; | \; \varphi_{e} > \lambda^*_{t+1} \hspace{2mm} \text{for some} \hspace{2mm} \varphi \in {\cal F}_{t+1} \bigg \}.$$
The edges in $E_{t} \backslash E_{t+1}$ are declared {\em inactive}, and the algorithm proceeds to the next value of $t$ if any
active edges remain.
As at least one edge becomes inactive in each step, the algorithm terminates in $O(|E|)$ steps.

It is often useful to think about this algorithm in a different, but equivalent way. First, observe that any edge whose
flow is {\em fixed} in every max-flow will carry exactly this amount in the outcome of the edge-fair algorithm
as well. Thus, we focus only on those edges $(i,j)$ with the property that $0 < \varphi_{ij} < u_{ij}$ for {\em some}
flow $\varphi \in {\cal F}^*(G,s,d,u)$. In particular,
from the observations in proposition~\ref{prop:pareto} on the set of
Pareto Optimal solutions, we could fix $z_{ij} = u_{ij}$ for $ij \in G(S_{-},D_{+})$ and 
$z_{ij} = 0$ for $ij \in G(S_{+},D_{-})$ and remove
these edges from the network. 
The reduced problem now decomposes into 2 
disjoint components: one in which the suppliers are rationed (and every demander
gets what they ask for), and the other in which the demanders are rationed, but
each supplier sends his entire supply. As the algorithm is completely symmetric,
we simply describe it for the case of rationed demanders. In this case each supplier
will be allocated his peak in every max-flow; and any flow that respects edge-capacities 
while allocating each supplier his peak, while allocating each demander no more than
his peak is a max-flow. Thus, the linear programming problem that must be solved in
each step can be explicitly described: the only edges that need to be considered
are those between $S_{+}$ and $D_{-}$. 

\begin{eqnarray*}
& \hspace{2mm} \text{Maximize} \hspace{3mm} \lambda_{t+1} \\
& \hspace{4mm} \text{subject to} \\
 &\hspace{10mm} \sum_{j} z'_{ij} = s_{i} \hspace{2mm} \forall \{i \in S_{+}, ij \in E_{t} \}  \\
 &\hspace{10mm} \sum_{i} z'_{ij} \leq d_{j} \hspace{2mm} \forall \{j \in D_{-}, ij \in E_{t} \} \\
&\hspace{10mm} \lambda_{t+1} \leq z'_{ij} \forall \{ij \in E_{t} \} \\
 &\hspace{10mm} u_{ij} \geq z'_{ij} \geq 0 
\end{eqnarray*}

Initially, every such edge is active, and
the algorithm tries to maximize the minimum amount carried by an active edge in
any max-flow.

\begin{theorem}
 \label{thm:edge_csy}
 The edge fair rule is symmetric, continuous, and consistent.
\end{theorem}

\textbf{Proof:} Symmetry follows because the rule is invariant, by definition, to permutations of supply nodes, demand nodes, and edge-capacities. Continuity is equally clear. 
Consistency is also immediate by the definition of the algorithm: we may assume that the edge $(i,j)$ that is dropped
to obtain the reduced problem is still present but carries a constant flow $z_{ij}$, where $z$ is the outcome chosen by
the rule for the original problem. Thus, the set of feasible solutions to the reduced problem is a subset of the set of
feasible solutions to the original problem at every stage of the algorithm: As the outcome $z$ for the original problem 
is a member of both sets, it will be chosen in both cases.
\qed


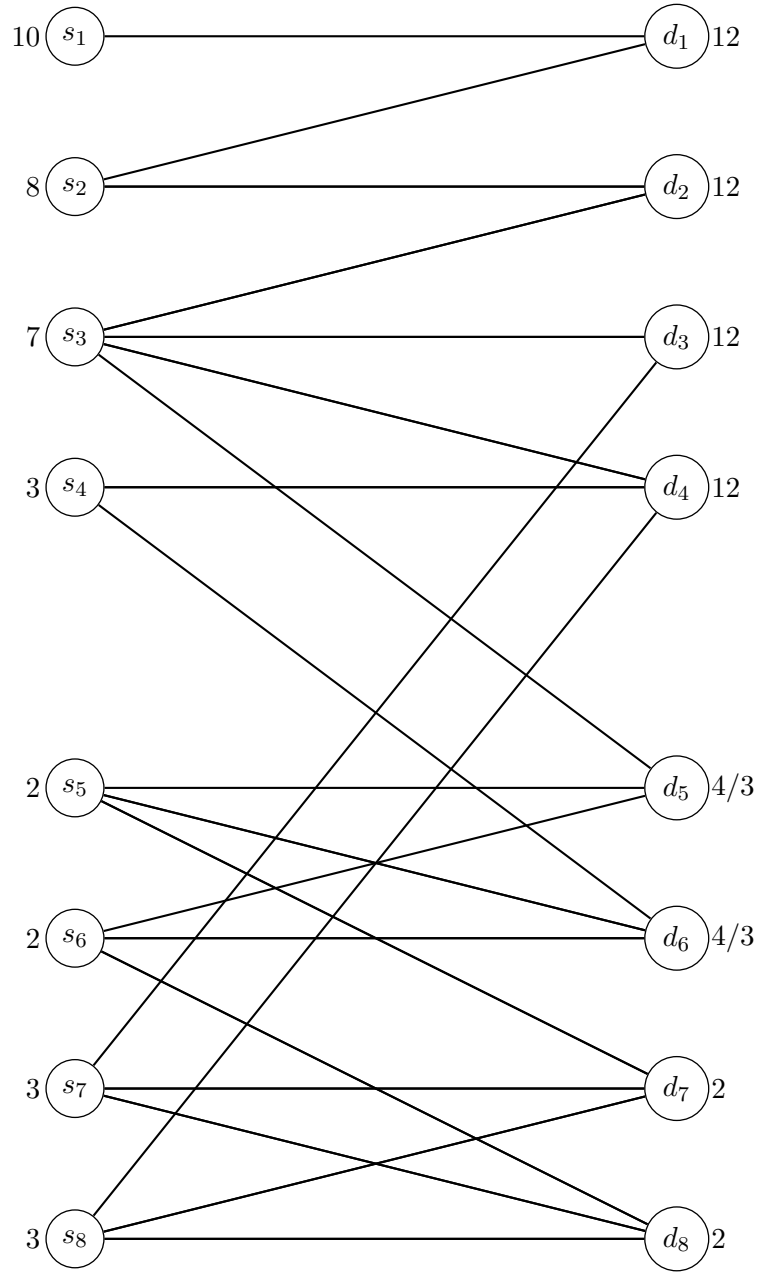
\begin{figure}
 \begin{center}
  \begin{tikzpicture}
   \GraphInit[vstyle=Normal]


	\Vertex[LabelOut,Lpos=180, x=0 ,y=0]{3}
 	\Vertex[x=0 ,y=0,Math,L=s_8]{s8}

 	\Vertex[LabelOut,Lpos=180, x=0 ,y=2]{3}
	\Vertex[x=0 ,y=2,Math,L=s_7]{s7}

	\Vertex[LabelOut,Lpos=180, x=0 ,y=4]{2}
	\Vertex[x=0 ,y=4,Math,L=s_6]{s6}

	\Vertex[LabelOut,Lpos=180, x=0 ,y=6]{2}
	\Vertex[x=0 ,y=6,Math,L=s_5]{s5}

	\Vertex[LabelOut,Lpos=180, x=0 ,y=10]{3}
 	\Vertex[x=0 ,y=10,Math,L=s_4]{s4}

 	\Vertex[LabelOut,Lpos=180, x=0 ,y=12]{7}
	\Vertex[x=0 ,y=12,Math,L=s_3]{s3}

	\Vertex[LabelOut,Lpos=180, x=0 ,y=14]{8}
	\Vertex[x=0 ,y=14,Math,L=s_2]{s2}

	\Vertex[LabelOut,Lpos=180, x=0 ,y=16]{10}
	\Vertex[x=0 ,y=16,Math,L=s_1]{s1}


	\Vertex[LabelOut,Lpos=0, x=8,y=0]{2}
 	\Vertex[x=8 ,y=0,Math,L=d_8]{d8}

	\Vertex[LabelOut,Lpos=0, x=8 ,y=2]{2}
 	\Vertex[x=8 ,y=2,Math,L=d_7]{d7}

	\Vertex[LabelOut,Lpos=0, x=8 ,y=4]{4/3}
 	\Vertex[x=8 ,y=4,Math,L=d_6]{d6}

	\Vertex[LabelOut,Lpos=0, x=8 ,y=6]{4/3}
 	\Vertex[x=8 ,y=6,Math,L=d_5]{d5}


	\Vertex[LabelOut,Lpos=0, x=8,y=10]{12}
 	\Vertex[x=8 ,y=10,Math,L=d_4]{d4}

	\Vertex[LabelOut,Lpos=0, x=8 ,y=12]{12}
 	\Vertex[x=8 ,y=12,Math,L=d_3]{d3}

	\Vertex[LabelOut,Lpos=0, x=8 ,y=14]{12}
 	\Vertex[x=8 ,y=14,Math,L=d_2]{d2}

	\Vertex[LabelOut,Lpos=0, x=8 ,y=16]{12}
 	\Vertex[x=8 ,y=16,Math,L=d_1]{d1}


	 \Edges(s8,d8,s8,d7,s8,d4)
         \Edges(s7,d7,s7,d8,s7,d3)
         \Edges(s6,d8,s6,d6,s6,d5)
       	 \Edges(s5,d7,s5,d6,s5,d5)

         \Edges(s4,d4,s4,d6)
         \Edges(s3,d4,s3,d3,s3,d2,s3,d5)
         \Edges(s2,d2,s2,d1)
       	 \Edges(s1,d1)

	\end{tikzpicture}
   \caption{Gallai-Edmonds Decomposition and the Edge Fair Allocation} 
   \label{fig:ged}
  \end{center}
\end{figure}


\paragraph{Example.}
We illustrate the algorithm on the problem shown in Figure~\ref{fig:ged} with 8 supply nodes and 8 
demand nodes.
All edges have infinite capacity except for the edges $(s7, d3)$ and $(s8, d4)$, which have capacity 0.5 each.
It is clear that these two capacitated edges must carry 0.5 unit of flow each in {\em every} max-flow, so
their flow can be fixed; by consistency, we could omit these edges from further consideration, and adjust 
the supplies
at $s7$ and at $s8$ and the demands at $d3$ and at $d4$ down by 0.5 unit each. Similarly,
the edges $(s3,d5)$ and $(s4,d6)$ carry no flow in any max-flow, and so can be omitted as well. The
problem now decomposes into two components: one involving the first 4 supply and demand nodes, where
the demand nodes are rationed in any max-flow; and the other involving the last 4 supply and demand nodes,
where the supply nodes are rationed in any max-flow.

First consider the problem involving the first four supply and demand nodes. Each supply node sends all
its supply, whereas each demand node receives at most what it wants.
The edge-fair algorithm applied to this problem gives the
following flow: first $z_{21} = 2$; then $z_{32} = z_{33} = z_{34} = 7/3$; then $z_{44} = 3$, after
which $z_{22} = 6$, and finally $z_{12} = 10$. The resulting allocation for the demanders in this
problem is $(12, 25/3, 7/3, 16/3)$; recall that demand nodes $3$ and $4$ also get 0.5 units of flow from
suppliers $s7$ and $s8$ respectively, so the final allocation for the demand nodes is $(12, 25/3, 17/6, 35/6)$.

Now consider the edge-fair algorithm applied to the last 4 supply and demand nodes. Here the supply nodes
are rationed whereas the demand is exactly met. It is easy to check that the
the edge-fair rule sends a flow of $2/3$ on each edge in this component so that the resulting allocation
for the supply nodes is $(2, 2, 4/3, 4/3)$; as the last 2 supply nodes also send 0.5 units of flow to the
other component, the final allocation for these supply nodes is $(2, 2, 11/6, 11/6)$.

To summarize, the edge-fair algorithm applied to this example results in an allocation of $(10, 8, 7, 3, 2, 2, 11/6, 11/6)$
for the supply nodes and $(12, 25/3, 17/6, 35/6, 4/3, 4/3, 2, 2)$ for the demand nodes. In contrast, it can
be verified that the egalitarian allocation results in an allocation of $(10, 8, 7, 3, 23/12, 23/12, 23/12, 23/12)$ for
the supply nodes, and $(10, 8, 11/2, 11/2, 4/3, 4/3, 2, 2)$ for the demand nodes.
This also highlights an important distinction between the edge-fair allocation and the egalitarian one: in our
example, demand nodes 3 and 4 have identical demands, and it is {\em possible} to give them the same allocation,
as shown by the Egalitarian allocation; the edge-fair rule, however, treats these demand nodes differently. In
particular, demand node 4 is better off under the edge-fair rule because of its improved connectivity.

%% file: model1.tex
\section{Model 1: Agents only on nodes}
\label{s:model1}

In this section we consider the version of the problem where
the nodes of the network are populated by agents. Specifically,
each supply node $i$ is occupied by a {\em supplier} $i$ and
each demand node $j$ is occipied by a {\em demander} $j$.
Thus, we our problem becomes one of transferring a single commodity
from a set $S$ of suppliers to a set $D$ of demanders using
the set $E$ of edges. The edge $e$ has a capacity $u_e$ that is known
to all the agents.
A transfer of the commodity from $S$ to $D$ is 
realized by a flow $\varphi$,
which specifies the amount of
the commodity transferred from supplier $i$ to demander $j$ using
the edge $(i,j) \in E$. The flow $\varphi$ induces an allocation
vector for each supplier and each demander as follows:
\begin{equation}
\text{for all }i\in S:x_{i}(\varphi )=\sum_{j\in f(i)}\varphi _{ij};\text{
for all }j\in D:y_{j}(\varphi )=\sum_{i\in g(j)}\varphi _{ij}  \label{1}
\end{equation}%
As we shall see in a moment, suppliers and demanders only care about
their {\em net} transfers, and not on how these transfers are distributed
across the agents on the other side.

Each supplier $i$ has \emph{single-peaked preferences}\footnote{%
Writing $P_{i}$ for agent $i$'s strict preference, we have for every $%
x_{i},x_{i}^{\prime }$: $x_{i}<x_{i}^{\prime }\leq s_{i}$ $\Rightarrow
x_{i}^{\prime }P_{i}x_{i},$ and $s_{i}\leq x_{i}<x_{i}^{\prime }$ $%
\Rightarrow x_{i}P_{i}x_{i}^{\prime }$.} $R_{i}$ (with corresponding
indifference relation $I_{i}$) over her \textit{net transfer} $x_{i}$, with
peak $s_{i}$, and each demander $j$ has single-peaked preferences $R_{j}$ ($%
I_{j}$) over her net transfer $y_{j}$, with peak $d_{j}$. 
We write $\mathcal{%
R}$ for the set of single peaked preferences over $\mathbb{R}_{+}$, and $%
\mathcal{R}^{S\cup D}$ for the set of preference profiles.

We think of the graph $G$ as fixed, and focus our attention on mechanisms
that elicit preferences from the agents and maps the reported preference
profile to a flow. For reasons that will be clear later, we focus
on allocation rules that are {\em peak} only: the flow (and hence the 
induced allocation vector for the suppliers and demanders) depends on
the reported preference profile of the agents only through their peaks.
Thus it makes sense to talk of the {\em problem} $(s,d)$: this 
emphasizes the fact that the {\em peaks} of the agents are private 
information whereas the other part of the problem (specifically, the
graph $G$ and the edge capacities $u$) are known.
To summarize: suppliers and demanders
report their peaks; the allocation rule is applied to
the graph $G$ with edge-capacities $u$, and the data $(s,d)$
where $s$ and $d$ are the reported peaks of the suppliers, 
and demanders. Our focus will be on mechanisms in
which no agent has an incentive to misreport his peak.

A mechanism is said to be {\em strategyproof}
if for any graph $G$ it is a weakly dominant strategy for an agent to 
truthfully report their peak. It is 
{\em group strategyproof} if for
any graph $G$ it is a weakly dominant strategy for any
coalition of agents to truthfully report their peaks. 

\subsection{Efficiency and Equity}

\underline{{\bf Pareto Optimality:}}
A feasible net transfer $(x,y)$ as defined in the previous section is Pareto Optimal if there
is no other allocation $(x',y')$ such that every agent is weakly better off and atleast
one agent is strictly better off in it. In mathematical terms, if $R_{i}$ and $I_{i}$
denote the preference and indifference relations respectively for agent $i$, then 
\begin{equation}
 \{ \forall \hspace{2mm} i,j: \hspace{2mm} x_{i}'R_{i}x_{i} \hspace{2mm} \text{and} \hspace{2mm} y_{j}'R_{j}y_{j}\} \implies \{ \forall \hspace{2mm} i,j: \hspace{2mm} x_{i}'I_{i}x_{i} \hspace{2mm} \text{and} \hspace{2mm} y_{j}'I_{j}y_{j}\}
\end{equation}

The following result shows that set of Pareto optimal transfers for peak-only rules
is intimately related to the set of max-flows.

\begin{prop}
\label{prop:pareto}
Fix the economy $(G,R)$. Let $S_{+}, S_{-}$ and $D_{+}, D_{-}$ be the Gallai-Edmonds
decomposition applied to the network $G$ with edge capacities given by $u$, supplies 
given by the peaks of the suppliers and the demands given by the peaks of 
the demanders. Then:
\begin{itemize}
\item[(a)] If the flow $\varphi$ implements Pareto optimal net transfers 
$(x,y)$, then:
\begin{equation}
  ij \in G(S_{-},D_{-}) \implies \phi_{ij} = u_{ij} ; \hspace{2mm}  
ij \in G(S_{+},D_{+} \cup (f(S_{-}) \cap D_{-}) ) \implies \phi_{ij} = 0 
\end{equation}
\item[(b)]
The transfers $(x,y)$ induced by a feasible flow $\varphi$ are Pareto optimal if
and only if
\begin{equation}
 x \geq s \hspace{2mm} \text{on} \hspace{2mm} S_{+}, \hspace{2mm} y\leq d \hspace{2mm} \text{on} \hspace{2mm} D_{-} \hspace{2mm} \text{and} \hspace{2mm} x_{S_{+}} = y_{D_{-}} - \phi(S_{-},D_{-})
\end{equation}
\begin{equation}
  x \leq s \hspace{2mm} \text{on} \hspace{2mm} S_{-}, \hspace{2mm} y\geq d \hspace{2mm} \text{on} \hspace{2mm} D_{+} \hspace{2mm} \text{and} \hspace{2mm} x_{S_{-}} = y_{D_{+}} + \phi(S_{-},D_{-})
\end{equation}
where $\phi(S_{-},D_{-})$ is the net flow from component $S_{-}$ to $D_{-}$. From earlier discussions, $\phi(S_{-},D_{-}) = \sum_{i\in S_{i}, j\in D_{-}} u_{ij}$
\end{itemize}
\end{prop}

We are particularly interested in Pareto optimal flows and transfers in which
no supplier or demander is allocated more than their peak: such solutions are
Pareto optimal for {\em any} single-peaked preferences of the agents as long
as the peaks are $s$ and $d$ respectively. Following Bochet et al., we call
this set $PO^*$ and note that $(x,y) \in PO^*$ if and only if $(x,y)$ is
Pareto optimal, $x \leq s$, and $y \leq d$. In particular, $(x,y) \in PO^*$ if
and only if
\begin{equation}
 x = s \hspace{2mm} \text{on} \hspace{2mm} S_{+}, \hspace{2mm} y\leq d \hspace{2mm} \text{on} \hspace{2mm} D_{-} \hspace{2mm} \text{and} \hspace{2mm} x_{S_{+}} = y_{D_{-}} - \phi(S_{-},D_{-})
\end{equation}
\begin{equation}
  x \leq s \hspace{2mm} \text{on} \hspace{2mm} S_{-}, \hspace{2mm} y = d \hspace{2mm} \text{on} \hspace{2mm} D_{+} \hspace{2mm} \text{and} \hspace{2mm} x_{S_{-}} = y_{D_{+}} + \phi(S_{-},D_{-})
\end{equation}

In the rest of the section, by a Pareto optimal solution we mean a flow inducing
net transfers $(x,y) \in PO^*$. We proceed now to discussions related to fairness. \\

\underline{\textbf{No Envy:}} A rule $(x,y) \in  \mathcal{F}(G,s,d,u)$ satisfies \emph{No Envy} if for any preference profile $R \in \mathcal{R}^{S \cup D}$
and $i,j \in
S$ such that $x_{j}P_{i}x_{i}$, there exists no $(x',y')$ such that 
\begin{eqnarray}
&x_{k} = x'_{k} \ \text{for all} \ k \in S \backslash \{ i,j \} ; \ y_{l} = y'_{l} \ \text{for all} \ l \in D \ \text{and} \\
&x'_{i}P_{i}x_{i} 
\end{eqnarray}
and a similar statement when we interchange the role of suppliers and demanders. \\

\underline{\textbf{Equal Treatment of Equals:}} A rule $(x,y) \in  \mathcal{F}(G,s,d,u)$ satisfies \emph{Equal Treatment of Equals} if for any preference profile $R \in \mathcal{R}^{S \cup D}$ and ${i,j} \in S$ such that $s_{i} = s_{j}$, if $x_{j} \neq
x_{i}$ then there exists no $(x',y')$ such that  
\begin{eqnarray}
 &x_{k} = x'_{k} \ \text{for all} \ k \in S \backslash \{ i,j \} ; \ y_{l} = y'_{l} \ \text{for all} \ l \in D \ \text{and} \\
&|x'_{i} - x'_{j}| < |x_{j} - x_{i}|
\end{eqnarray}
and a similar statement when we interchange the role of suppliers and demanders. 

If an allocation rule always results in a Pareto optimal allocation and satisfies No Envy, then 
it also satisfies Equal Treatment of Equals (Refer to Proposition 5 in Bochet et al.~\cite{bims}). 

The egalitarian rule of Bochet et al.~\cite{bims} is a selection from the Pareto set $PO^*$ as is the edge-fair allocation rule.  They  also show that the egalitarian rule is envy-free but the  inconsistency of the rule follows from figure~\ref{fig:cons-noenvy} where we remove the node $d_{1}$ from the network on the left. The egalitarian allocation on the reduced 
network improves the allocation of $s_{2}$ by sending 1 unit of flow on the edge $s_{2} - d_{2}$

\begin{figure}
 \begin{center}
  \begin{tikzpicture}
   \GraphInit[vstyle=Normal]



	\Vertex[LabelOut,Lpos=180, x=0 ,y=0]{3}
 	\Vertex[x=0 ,y=0,Math,L=s_2]{s2}

 	\Vertex[LabelOut,Lpos=180, x=0 ,y=2]{3}
	\Vertex[x=0 ,y=2,Math,L=s_1]{s1}


	\Vertex[LabelOut,Lpos=0, x=4,y=0]{2}
 	\Vertex[x=4 ,y=0,Math,L=d_2]{d2}

	\Vertex[LabelOut,Lpos=0, x=4 ,y=2]{2}
 	\Vertex[x=4 ,y=2,Math,L=d_1]{d1}

     \Edges[label={$z_{11}=2$}](s1,d1)

     \Edges[label={$z_{22}=2$}](s2,d2)
     \Edges(s2,d1)




	\Vertex[LabelOut,Lpos=180, x=6 ,y=2]{3}
 	\Vertex[x=6 ,y=2,Math,L=s_1]{s1}

	\Vertex[LabelOut,Lpos=0, x=10 ,y=2]{2}
 	\Vertex[x=10 ,y=2,Math,L=d_1]{d1}

     \Edges(s1,d1)




	\Vertex[LabelOut,Lpos=180, x=12 ,y=2]{3}
 	\Vertex[x=12 ,y=2,Math,L=s_1]{s1}

 	\Vertex[LabelOut,Lpos=180, x=12 ,y=0]{1}
	\Vertex[x=12 ,y=0,Math,L=s_2]{s2}


	\Vertex[LabelOut,Lpos=0, x=16,y=2]{2}
 	\Vertex[x=16,y=2,Math,L=d_1]{d1}

\Edges[label={$z_{11}=1$}](s1,d1)
\Edges[label={$z_{21}=1$}](d1,s2)

	\end{tikzpicture}
   \caption{Inconsistency of the egalitarian rule and envy of edge fair rule} 
   \label{fig:cons-noenvy}
  \end{center}
\end{figure}
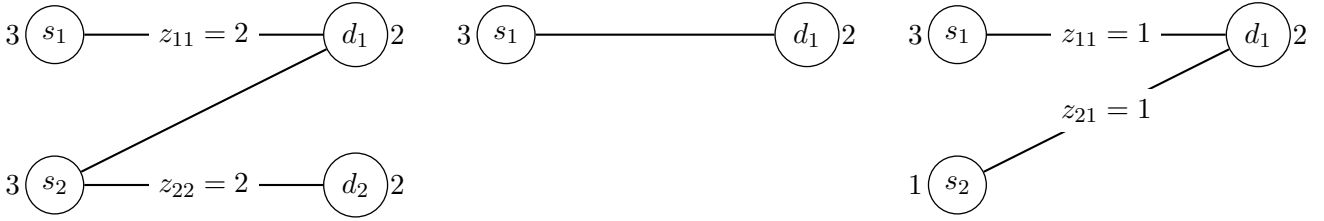


We have already seen that the edge-fair rule is also consistent. But here is an exmple
where the edge-fair rule has envy. But one can show that no consistent rule is envy-free (under $PO^*$) using the same example.

\begin{lemma}
 There is no mechanism which is simultaneously envy free for agents on the nodes and edge consistent under $PO^{*}$
\end{lemma}

Let us consider the same example as in figure~\ref{fig:cons-noenvy}. Suppose the mechanism is envy free:
Any envy free solution should allocate 2 units to each supplier 1 and 2. 
 This establishes a unique edge flow: $(z_{11},z_{21},z_{22}) = (2,0,2)$.
Lets remove the edge $s_{2}-d_{2}$ with $z_{22} =2$ units allocation. If this mechanism was also consistent, then on this
reduced network the mechanism should have an allocation $(z_{11},z_{21}) = (2,0)$ on the edges. But the no-envy solution on this reduced graph would allocate $(z_{11},z_{21}) = (1,1)$

Now, suppose the given mechanism is edge consistent and lets say $(z_{11},z_{21},z_{22}) = (2,0,2)$ is an allocation from some edge consistent rule. Removing the edge $s_{2}-d_{2}$, in the reduced graph the allocation from an edge consistent mechanism is $(z_{11},z_{21}) = (2,0)$ but this does not allocate a envy free solution for the nodes on the reduced graph.
As a consequence, if the mechanism is edge consistent it cannot allocate $(z_{11},z_{21},z_{22}) = (2,0,2)$ in the original network but this is the only envy free solution on that network. The same example also shows that any edge consistent mechanism violates the property equal treatment of equals. \qed

These results imply that no rule can be Pareto efficient, Envy-free and Consistent.
Both the egalitarian and edge-fair rules find Pareto efficient allocations; where
they differ is that the egalitarian rule relaxes consistency but is envy-free, but
the edge-fair rule relaxes envy-freeness but is consistent.
Bochet et al.~\cite{bims} show that the egalitarian rule is strategyproof; and in
our earlier paper, we show that it is in fact group strategyproof. A natural question
is if the edge-fair rule enjoys these properties as well. We answer this question in
the affirmative in the next section.

\subsection{Strategic Issues}
We start with a formal definition of strategy-proofness and group strategy-proofness.
Informally, a mechanism is strategyproof if it is a (weakly) dominant strategy for 
each agent to reveal his peak truthfully; and a mechanism is group strategyproof 
if it is a (weakly) dominant strategy for any group of agents to reveal their
peaks truthfully.

\textbf{Strategyproof:} A rule $(x,y)$ on $(G,s,d)$ is \emph{strategyproof} if 
for all $R \in R^{S \cup D}$, $i \in S, j \in D$ and $R_{i}', R_{j}' \in \mathcal{R}$
\begin{equation}
 x_{i}(R)R_{i}x_{i}(R'_{i},R_{-i}) \hspace{2mm} 
\text{and} \hspace{2mm} y_{j}(R)R_{j}y_{j}(R_{j}',R_{-j})
\end{equation}

\textbf{Peak Group Strategyproof:} A rule $(x,y)$ on $(G,s,d)$ is 
\emph{group strategyproof} if for all
$R \in R^{S \cup D}$, $A \subseteq S \cup D$ and $R_{i}' \in \mathcal{R}$
\begin{equation}
 x_{i}(R)R_{i}x_{i}(R'_{A},R_{-A}) \hspace{2mm} \forall i \in A 
\end{equation}

These properties are closely related to an {\em invariance} property that
we formally define below:

\textbf{Invariance}: For all 
$R \in \mathcal{R}^{S \cup D}$, $i \in S$ and $R_{i}' \in \mathcal{R}$
\label{def:inv}
\begin{eqnarray}
 \{ s[R_{i}] < x_{i}(R) \hspace{2mm} \text{and} \hspace{2mm} s[R_{i}'] \leq x_{i}(R) \} \hspace{2mm} \text{or}
 \hspace{2mm} \{s[R_{i}] > x_{i}(R) \hspace{2mm} \text{and} \hspace{2mm} s[R_{i}'] \geq x_{i}(R)\} \\
 \implies x_{i}(R'_{i},R_{-i}) = x_{i}(R)
\end{eqnarray}
and a similar invariance property can be defined with respect to the demanders.

\begin{lemma}
 \label{lem:sp_inv} 
For any rule that always selects an allocation $(x,y) \in PO^{*}$, 
strategyproofness and invariance are equivalent.
\end{lemma}

\textbf{Proof:} First we show that, under $PO^*$, strategyproofness implies
invariance: As the allocation is in $PO^{*}$ we have $x_{i} \leq s_{i}$.
Thus, to prove invariance we need to show that when 
$x_{i} < s_{i}$, and $s_{i}' \geq x_{i}$ we have $x'_{i} = x_{i}$. 
Suppose not and we have $x'_{i} < x_{i}$. Then agent $i$ benefits 
by misreporting his peak as $s_{i}$ when his true peak is $s_{i}'$,
which violates strategyproofness. Similarly, if $x_{i}' > x_{i}$, we can 
construct a profile $R^{*}$ such that $x_{i}'P_{i^{*}}x_{i}$. As a 
$PO^{*}$ + Strategyproof rule is peak-monotonic and as a consequence
own peak only (Bochet et al.~\cite{bims}), $x_{i}(R_{i}^{*},R_{-i}) = x_{i}(R)$. 
Hence, $i$ benefits by sreporting $s_{i}'$ when his true 
peak is $s_{i}$, which violates strategyproofness again.

We now show the converse. Suppose the rule is not strategyproof.
Under a $PO^{*}$ rule, $x_{i} = s_{i}$ for every agent
$i \in S_{+}$, hence those agents never misreport. Every agent in $i \in S_{-}$ is such that $x_{i} \leq s_{i}$. So, any agent who deviates
and improves his allocation is such that $s'_{i} \geq x_{i} < s_{i}$ and $x_{i}' P_{i} x_{i}$. But this is not possible under an invariant rule.
Hence, the rule is indeed strategyproof. \qed

As we discussed earlier, the egalitarian rule is strategyproof, but is also group
strategyproof. A natural question is if every strategyproof rule in our problem is
also group strategyproof\footnote{Barbera et al.~\cite{barberagsp} study environments where this is
indeed the case.}. As it turns out, the answer is "no" as shown by the following example.
Consider the following mechanism, if the report of $d_{0} >=5$, then apply the 
egalitarian mechanism and if the report of $d_{0} < 5$, follow the edge fair mechanism. 
This rule is clearly strategyproof. But agent $d_{0}$ and $s_{1}$ can collude such that agent $d_{0}$ misreports his peak as 4 (when his/her true peak is 6). This improves the allocation
of agent $s_{1}$ by 1 unit, keeping the allocation of $d_{0}$ to be the same.


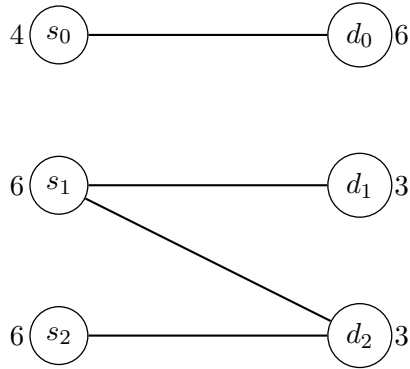
\begin{figure}[h!]\label{AA}
 \begin{center}
  \begin{tikzpicture}
   \GraphInit[vstyle=Normal]



	\Vertex[LabelOut,Lpos=180, x=0 ,y=0]{6}
 	\Vertex[x=0 ,y=0,Math,L=s_2]{s2}

 	\Vertex[LabelOut,Lpos=180, x=0 ,y=2]{6}
	\Vertex[x=0 ,y=2,Math,L=s_1]{s1}

	\Vertex[LabelOut,Lpos=180, x=0 ,y=4]{4}
	\Vertex[x=0 ,y=4,Math,L=s_0]{s0}


	\Vertex[LabelOut,Lpos=0, x=4,y=0]{3}
 	\Vertex[x=4 ,y=0,Math,L=d_2]{d2}

	\Vertex[LabelOut,Lpos=0, x=4 ,y=2]{3}
 	\Vertex[x=4 ,y=2,Math,L=d_1]{d1}

	\Vertex[LabelOut,Lpos=0, x=4,y=4]{6}
 	\Vertex[x=4 ,y=4,Math,L=d_0]{d0}

     \Edges(s0,d0)
     \Edges(s1,d1)
     \Edges(s1,d2)
     \Edges(s2,d2)

	\end{tikzpicture}
   \caption{Invariance and GSP are not equivalent} 
  \end{center}
\end{figure}


We know from Bochet et al.~\cite{bims} that strategyproofness of a 
rule can just be characterized by peak monotonicity and invariance. 
From the above discussion, strategyproofness is characterized by $PO^{*}$ 
and invariance. So, the natural question is what other additional property 
is needed to make a mechanism groupstrategyproof. Next, we show that 
any groupstrategyproof mechanism can be characterized by $PO^{*}$ and
the following {\em stronger} invariance property:

\textbf{Strong Invariance}: For all $R \in \mathcal{R}^{S \cup D}$, $i \in S$ and $R_{i}' \in \mathcal{R}$
\begin{eqnarray}
 \{ s[R_{i}] < x_{i}(R) \hspace{2mm} \text{and} \hspace{2mm} s[R_{i}'] \leq x_{i}(R) \} \hspace{2mm} \text{or}
 \hspace{2mm} \{s[R_{i}] > x_{i}(R) \hspace{2mm} \text{and} \hspace{2mm} s[R_{i}'] \geq x_{i}(R)\} \\
 \implies x_{j}(R'_{i},R_{-i}) = x_{j}(R) \hspace{2mm} \forall j \in S \ \text{and} \\
  y_{l}(R'_{i},R_{-i}) = y_{l}(R) \hspace{2mm} \forall l \in D
\end{eqnarray}
and a similar strong invariance property can be defined with respect to the demanders. 

In other words, while
invariance implies that the allocation of a supplier is unchanged whenever 
his peak misreport is above his allocation, strong invariance implies that 
the allocation of {\em every} agent is unchanged when a particular agent misreports his
peak over his current allocation. 


Our main result in this section is the following.

\begin{theorem}
\label{thm:gsp_peak}
Any mechanism that always selects an allocation from $PO^*$ satisfies strong 
invariance if and only if it is group strategy-proof.
\end{theorem}

\textbf{Proof of theorem~\ref{thm:gsp_peak}} 
We follow the proof technique introduced in Chandramouli $\&$ Sethuraman~\cite{cs} for the first part of the theorem 
$PO^{*}$, strong invariance $\implies$ Peak GSP.

Suppose such a rule is not peak group strategyprof then lets focus on a network $G$ with the
{\em smallest} number of nodes. Suppose the true peaks
of the suppliers and demanders are $s$ and $d$ respectively,
and suppose their respective misreports are $s'$ and $d'$.
We can assume that $d_j > 0$ for every 
demander $j$, as otherwise deleting $j$ would result in a smaller counterexample.
Fix a coalition $A$ of suppliers and a coalition $B$ of
demanders : note that $A$ contains
all the suppliers $k$ with $s'_k \not = s_k$, and $B$
includes all demanders $\ell$ with $d'_\ell \not = d_\ell$.

Let $(x,y)$ and $(x',y')$ be the respective allocations
to the suppliers and demanders when they report $(s,d)$ and $(s',d')$ 
respectively. 
Let $S_{+}, S_{-}, D_{+}, D_{-}$ be the decompsition 
when the agents report $(s,d)$, and let $S'_{+}, S'_{-},
D'_{+}, D'_{-}$ be the decomposition when the agents 
report $(s',d')$.
We shall show that when
the agents report $(s',d')$ rather than $(s,d)$,
the only allocation in which each agent in $A \cup B$
is (weakly) better off, then $x'_k = x_k$ for 
all $k \in A$ and $y'_{\ell} = y_{\ell}$ for all $\ell \in B$.
This establishes the required contradiction.

Let $Y' := D_{+} \cap D'_{-}$. If $Y' = \{\emptyset\}$, then consider the set of suppliers $S_{-} \cap S_{+}'$.
Every supplier $i \in S_{-} \cap S_{+}'$ do not send flow to any demander $j$ in $D_{+}'$. Hence, these suppliers
can send flow to only demanders in $f(S_{-} \cap S_{+}') \cap D_{-}'$. Now observe, $z_{ij} = u_{ij}, z_{ij}' \leq u_{ij}$ when 
the reports are $s$ and $s'$ respectively for every agent $i \in S_{-} \cap S_{+}' , j \in f(S_{-} \cap S_{+}') \cap D_{-}'$. 
Hence, every supplier $i \in S_{-} \cap S_{+}'$ sends weakly less flow to every agent
connected to him. Hence, $s_{i}' = x_{i}' \leq x_{i} \leq s_{i}$. So, we can conclude $A = \{\emptyset\}$ when $Y' = \{\emptyset\}$.\\

We now consider the case $Y' \neq \{\emptyset\}$ and make observations about the suppliers $X':= g(Y') \cap S_{-} \cap S'_{+}$. 
Let $Y'':= f(X') \cap D_{-}' \cap D_{-}$

\begin{itemize}
\item
For any such supplier $k$, $s'_k = x'_k$ and $x_k \leq s_k$. Also,
$d_\ell = y_\ell$ and $y'_\ell \leq d'_\ell$ for any $\ell \in Y'$.
\item
When the report is $s'$, 
every such supplier can send flow only to the demanders 
in $Y' \cup Y''$: this is because no link exists between agents in $X'$ and demanders in $D_{-}' \backslash \{Y' \cup Y''\}$ and 
$z_{ij} = 0$ $\forall$ $ij \in G(S_{+}',D_{+}')$ in a pareto optimal allocation. Also, observe that $z_{ij} \leq u_{ij}$ $\forall$
$ij \in G(X',Y'')$ and $z_{ij} = u_{ij}$ $\forall$ $ij \in G(S_{-}',Y')$.
Therefore
$\sum_{k \in X'} x'_k \; \leq \; \sum_{\ell \in Y'} y'_\ell - \sum_{ij \in G(S_{-}',Y')}{u_{ij}} + \sum_{ij \in G(X',Y'')}{u_{ij}} $
\item
When the report is $s$, $z_{ij} = u_{ij}$ $\forall$ $ij \in G(X',Y'')$. The agents in $Y'$ can receive flow only from agents in $X'$ and $g(Y') \cap S_{-}' \cap S_{-}$.
The agents in $Y'$ can receive at most $\sum_{ij \in G(S_{-}',Y')}{u_{ij}}$ units of flow from the suppliers $g(Y') \cap S_{-}' \cap S_{-}$. Hence, the remaining
allocation has to be supplied from $X'$. Also, note that $f(X') \supseteq Y'$.
Therefore
$\sum_{k \in X'} x_k \geq \sum_{\ell \in Y'} y_\ell - \sum_{ij \in G(S_{-}',Y')}{u_{ij}} + \sum_{ij \in G(X',Y'')}{u_{ij}}$.
\end{itemize}
Let $ f(S_{-}',Y'): = - \sum_{ij \in G(S_{-}',Y')}{u_{ij}} + \sum_{ij \in G(X',Y'')}{u_{ij}}$.
Finally, note that $s'_k = s_k$ for all $k \not \in A$, and
$d'_\ell = d_\ell$ for all $\ell \not \in B$.
These observations first lead to
\begin{equation}
\label{eq:2glbg}
\sum_{\substack{ k \in X'  \\ k \not \in A} } s_k + 
\sum_{\substack{k \in X'  \\ k \in A}} x'_k  \; = \;
   \sum_{\substack{ k \in X' \\ k \not \in A} } s'_k +
\sum_{\substack{k \in X'  \\ k \in A}} x'_k  \; = \;
    \sum_{k \in X' } x'_k \; \leq \; \sum_{\ell \in Y' } y'_\ell + f(S_{-}',Y')
\end{equation}
Note that every demander $\ell$ in $Y' \cap B$ receives {\em exactly} his peak
allocation $d_\ell$ for a truthful report, so for the coalition $B$ of 
demanders to do weakly better in the $(G,s',d')$ problem, 
$y'_\ell = d_\ell$ for each such $\ell$.
Therefore,
\begin{equation}
\label{eq:2gs}
\sum_{\ell \in Y'} y'_\ell \; = \; 
     \sum_{\ell \in Y' \setminus B} y'_\ell + \sum_{\ell \in Y' \cap B} y'_\ell \; \leq \; 
     \sum_{\ell \in Y' \setminus B} d'_\ell + \sum_{\ell \in Y' \cap B} d_\ell \; = \; 
     \sum_{\ell \in Y'} d_\ell.
\end{equation}
Finally,
\begin{equation}
\label{eq:2gubg}
\sum_{\ell \in Y'} d_\ell + f(S_{-}',Y')  \; = \; 
\sum_{\ell \in Y'} y_{\ell} \ + f(S_{-}',Y')
\; \leq \;
\sum_{k \in X' } x_k \; \leq \; 
\sum_{\substack{ k \in X' \\ k \not \in A} } s_k + 
\sum_{\substack{k \in X' \\ k \in A}} x_k
\end{equation}

\bigskip

For every supplier in $A$ to be (weakly) better off when 
reporting $s'$, we must have
$x'_k \geq x_k$ for each $k \in X'$. Combining this with
inequalities (\ref{eq:2glbg}) and (\ref{eq:2gubg}), we conclude that 
all the inequalities in (\ref{eq:2glbg})-(\ref{eq:2gubg}) hold as
equations.
In particular, $x'_k = x_k$ for all $k \in X' $, and
$y'_{\ell} = y_{\ell}$ for $\ell \in Y'$.
Therefore, whether the report is $s$ or is $s'$, the suppliers in 
$X'$ send all of their
flow only to the demanders in $Y'$ and $Y''$; Moreover, the edges from $X'$ to $Y''$ and 
$S_{-}'$ to $Y'$ are saturated and
that the demanders in $Y'$
receive all of their flow only from the suppliers in $X'$ and from the saturated edges from 
$S_{-}'$
Therefore, removing the suppliers in $X'$ and the 
demanders in $Y'$ and the saturated edges from $X'$ to $Y''$ and 
$S_{-}'$ to $Y'$
does not affect the allocation rule for either problem. As we picked a
smallest counterexample, $Y'$ must be empty.

We now turn to the other case.
Let $\tilde{X} := S_{+} \cap S'_{-}$.
Define  $\tilde{\tilde{{Y}}}:= f(\tilde{X}) \cap D_{-} \cap D_{-}'$ and 
Consider the demanders in $\tilde{Y}:= f(\tilde{X}) \cap D_{-} \cap D_{+}'$ 

\begin{itemize}
\item
For any such demander $\ell$, $d'_\ell = y'_\ell$ and $y_\ell \leq d_\ell$. 
Also, $s_k = x_k$ and $x'_k \leq s'_k$ for any $k \in \tilde{X}$.
\item
When the report is $s'$, every such demander can receive flow from
the suppliers in $\tilde{X}$ and suppliers in $g(\tilde{Y}) \cap S_{-} \cap S_{-}'$.
The supplier $i \in \tilde{X}$ send flow $z_{ij} = u_{ij}$ to every demander $j \in \tilde{Y}$ in the graph $G(\tilde{X},\tilde{\tilde{Y}})$.
Suppliers in $S_{-}$ send at most $\sum_{ij \in G(S_{-},\tilde{Y})} u_{ij}$ units of flow to $\tilde{Y}$. But note that $f(\tilde{X}) \supseteq \tilde{Y}$
and hence $\tilde{X}$ can send flow to agents in $D_{+}' \backslash \tilde{Y}$.
Therefore
$\sum_{k \in \tilde{X}} x'_k \; \geq \; 
	\sum_{\ell \in \tilde{Y}} y'_\ell -\sum_{ij \in G(S_{-},\tilde{Y})} u_{ij} + \sum_{ij \in G(\tilde{X},\tilde{\tilde{Y}})} u_{ij}$.
\item
When the report is $s$, the suppliers in $\tilde{X}$ send flow
only to the demanders in $D_{-}$, and they can send flow only to the
demanders they are connected to. so the suppliers in $\tilde{X}$ can
send flow only to the demanders in $\tilde{Y} \cup \tilde{\tilde{Y}}$. The agents in $\tilde{X}$
can send at most $\sum_{ij \in G(\tilde{X},\tilde{\tilde{Y}})} u_{ij}$ units of flow to the agents
in $\tilde{\tilde{Y}}$. Also, the agents in $\tilde{Y}$ receive flow $\sum_{ij \in G(S_{-},\tilde{Y})} u_{ij}$  from $S_{-}$
Therefore
$\sum_{k \in \tilde{X}} x_k \; \leq \; 
	\sum_{\ell \in \tilde{Y}} y_\ell -\sum_{ij \in G(S_{-},\tilde{Y})} u_{ij} + \sum_{ij \in G(\tilde{X},\tilde{\tilde{Y}})} u_{ij}$.
\end{itemize}
Lets denote $\tilde{f}(S_{-},\tilde{Y}):= -\sum_{ij \in G(S_{-},\tilde{Y})} u_{ij} + \sum_{ij \in G(\tilde{X},\tilde{\tilde{Y}})} u_{ij} $.

Finally, note that $s'_k = s_k$ for all $k \not \in A$, and
$d'_\ell = d_\ell$ for all $\ell \not \in B$.
Putting all this together, we have:
\begin{equation}
\label{eq:2glbbg}
\sum_{\substack{ \ell \in \tilde{Y} 
\\ \ell \not \in B} } d_\ell + 
\sum_{\substack{ \ell \in \tilde{Y} 
\\ \ell \in B} } d'_\ell + \tilde{f}(S_{-},\tilde{Y})
 \; = \;
\sum_{\ell \in \tilde{Y}} d'_\ell + \tilde{f}(S_{-},\tilde{Y}) \; = \;
\sum_{\ell \in \tilde{Y}} y'_\ell + \tilde{f}(S_{-},\tilde{Y})
\end{equation}
and
\begin{equation}
\label{eq:2gaux}
\sum_{\ell \in \tilde{Y}} y'_\ell + \tilde{f}(S_{-},\tilde{Y})\; \leq \;
\sum_{k \in \tilde{X}} x'_k  \; \leq \;
\sum_{k \in \tilde{X} \setminus A} s'_k + 
	\sum_{k \in \tilde{X} \cap A} x'_k \; = \; 
\sum_{k \in \tilde{X} \setminus A} s_k + 
\sum_{k \in \tilde{X} \cap A} x'_k.
\end{equation}
Note that every supplier $k$ in $\tilde{X} \cap A$ receives {
\em exactly} his peak
allocation $s_k$ for a truthful report, so for the coalition $A$ of 
suppliers to do weakly better in the $(G,s',d')$ problem, 
$x'_k = s_k$ for each such $k$.
Thus,
\begin{equation}
\label{eq:2gubbg}
\sum_{k \in \tilde{X} \setminus A} s_k + 
\sum_{k \in \tilde{X} \cap A} x'_k \; = \;
\sum_{k \in \tilde{X}} s_k  \; = \;
\sum_{k \in \tilde{X}} x_k \; \leq \;
\sum_{\ell \in \tilde{Y} } y_\ell + \tilde{f}(S_{-},\tilde{Y}) \; \leq \;
\sum_{\substack{ \ell \in \tilde{Y} 
\\ \ell \not \in B} } d_\ell + 
\sum_{\substack{ \ell \in \tilde{Y} 
\\ \ell \in B} } y_\ell +\tilde{f}(S_{-},\tilde{Y})
\end{equation}

\bigskip
For every demander in $B$ to be (weakly) better off, we must have
$y'_{\ell} \geq y_{\ell}$ for each $\ell \in \tilde{Y}$.
Combining this with
inequalities (\ref{eq:2glbbg})-(\ref{eq:2gubbg}), we conclude that 
all the inequalities in (\ref{eq:2glbbg})-(\ref{eq:2gubbg}) hold as
equations.
In particular, $x'_k = x_k$ for all $k \in \tilde{X}$, and
$y'_{\ell} = y_{\ell}$ for $\ell \in \tilde{Y}$.
Therefore, whether the report is $s$ or is $s'$, the suppliers in 
$\tilde{X}$ send all of their
flow only to the demanders in $\tilde{Y}$ and to the demanders in $\tilde{\tilde{Y}}$; Moreover,
the edges from $\tilde{X}$ to $\tilde{\tilde{Y}}$ are saturated in both problems; So are the 
edges $S_{-}$ to $\tilde{Y}$.
and that the demanders in $\tilde{Y}$
receive all of their flow only from the suppliers in $\tilde{X}$ and through the 
saturated edges from $S_{-}$ in both the problems.
Therefore, removing the suppliers in $\tilde{X}$ and the 
demanders in $\tilde{Y}$ and the saturated edges from $\tilde{X}$ to $\tilde{\tilde{Y}}$ and
$S_{-}$ to $\tilde{Y}$, we 
do not affect the allocation rule for either problem. As 
we picked a smallest counterexample, $\tilde{X}$ must be empty.

We now establish that the decomposition does not change in a smallest
counterexample. We already know that $Y' = \emptyset$, which implies
$D'_{-} \subseteq D_{-}$. Suppose this containment is strict so that
there is a demander $j \in D_{-} \setminus D'_{-}$. The links from $S_{-}$ to $j$ 
are completely saturated. As $\tilde{X} = \emptyset$, $j$ receives flow only from the 
suppliers in $S_{-} \cap S_{-}'$. Also, the flow on the edges from a supplier $ i \in S_{-} \cap S_{-}'$ 
to $j$ is such that $z'_{ij} \leq u_{ij} = z_{ij}$. Hence, the allocation for agent $j$ is such that, 
$y'_{j} = d'_{j} \leq y_{j}$. But now note that, if $j \in B$ then, $d'_{j} \geq y_{j}$ or if $j \notin B$
then $y'_{j} = d_{j} \leq y_{j} \leq d_{j}$. In both the cases, we have the equality $y'_{j} = d'_{j} = y_{j}$.
This implies, $g(j) \cap S_{-} \cap S_{+}' = \{\emptyset\}$; The links from $S_{-}$ to $j$ is saturated in both 
the problems (Follows from the fact that the given rule allocates the pareto value to the agents in both the 
networks, in particular $y_{k}' = d_{k}'$ when the reports are $d'$).
Hence, we can remove those saturated edges and adjust the peaks of suppliers and demanders.
The adjusted demand of agent $j$ now is $d'_{j} = 0$. w.l.o.g we can skip the case $d'_{j} = 0$ as we can delete such a $j$
 to obtain the new decomposition or just place it in $D_{-}$. Therefore $D'_{-}
= D_{-}$, which implies $D'_{+} = D_{+}$, $S'_{+} = S_{+}$, and $S'_{-} =
S_{-}$. 

To complete the argument, let $A$ be as defined earlier.
Let $A_{+} = A \cap S_{+}$ and $A_{-} = A \cap S_{-}$, $B_{+} = A \cap D_{+}$
and $B_{-} = A \cap D_{-}$. Now, for any $j \in B_{+}$, $d'_{j} \neq d_{j}$
implies $y'_{j} = d'_{j} \neq d_{j}$ causing $j$ to do worse by reporting
$d'_{j}$. Hence, it follows, $\forall j \in B_{+}$, $d'_{j} = d_{j}$. By a
similar argument, we could establish $s'_{j} = s_{j} \forall j \in A_{+}$.
 
For any $i \in A_{-}$, $s'_i < x_i$ implies $x'_i \leq s'_i < x_i$, causing
$i$ to do worse by reporting $s'_i$.
Likewise, any $i \in B_{-}$, $d'_i < y_i$ implies $y'_i \leq d'_i < y_i$,
causing $i$ to do worse by reporting $d'_i$. So any improving coalition $A$ must
be such that $s'_i \geq x_i$ for all $i \in A_{-}$ and
$d'_i \geq y_i$ for all $i \in B_{-}$.

Now, we use the strong invariance property of the rule to conclude the result. 
Partition the agents in $A_{-} = A_{s} \cup A_{x}$ where $A_{s}:= \{x_{i} = s_{i} | i \in A_{-} \}$ and $A_{x}:=\{x_{i} < s_{i} | i \in A_{-} \}$.
Lets start with an agent $i \in A_{s}$, such an agent reports $s'_{i} > x_{i} = s_{i}$ and receives $x'_{i} = s_{i}$. Now, consider the alternate
set of reports such that $s''_{j} = s_{j}'$ for all agents $j \neq i$ and $s''_{i} = s_{i}$ and denote the corresponding network by $G(S'',D'')$.
Strong invariance property implies that when the peak report $s_{ii}' \geq x'_{i}=s_{i}$ then the allocation profile of the agents remains the same 
in the networks $G(S',D')$ and $G(S'',D'')$. Hence, we can find a \emph{smaller} counterexample by removing $i$ from $A_{-}$. Hence, we can remove
all the agents from $A_{s}$ and still find a smaller counterexample. Hence, we can assume the smallest counterexample $A_{s} = \{\emptyset\}$.

On similar lines, strong invariance property also implies that when an agent $i$ with $x_{i} < s_{i}$ misreports such that $s_{i}' > x_{i}$ then 
$x'_{i} = x_{i}$ $\forall$ $i \in S$. Hence, applying this argument for each agent iteratively, we can conclude that when the set of agents 
in $A_{x}$ inflate their peaks, the allocation does not change i.e. $x'_{i} = x_{i}$ $\forall$ $i \in S$. Hence, no agent improves his allocation 
under this rule, concluding the result. 
\qed

Now, we turn to prove the other direction of the result i.e. any rule that is $PO^{*}$ and peak GSP is strongly invariant. We discuss the 
result only for the suppliers, by symmetry a similar reasoning follows for the demanders. Suppose such a rule is not strongly invariant. 
Since agents in $S_{+}$ receive their peak, strong
invariance property needs to be discussed only in the context of the agents in $S_{-}$ where $x_{i} \leq s_{i}$. Now, consider an agent 
$i \in S_{-}$ such that $x_{i} < s_{i}$. Consider a report by agent $i$ such that $s'_{i} \geq x_{i}$. From Lemma~\ref{lem:sp_inv} it follows that
$PO^{*}$ + strategyproof implies invariance. Hence, $x'_{i} = s_{i}$. Furthermore, it follows from the earlier discussion
that the decomposition and maximum flow does not change in this new problem. Hence, $\sum_{k \in S_{-}} x_{k} = \sum_{k \in S_{-}} x'_{k}$. Suppose
$x'_{k} = x_{k}$ $\forall$ $k \in S_{-}$ then we are done. Suppose, $x'_{k} \neq x_{k}$ for some agent $k \in S_{-}$, then there exists at least one
agent $j$ such that $s_{j} \geq x'_{j} > x_{j}$ (agent $j$ improves the allocation). Thus, the pair of agents $i$ and $j$ represent a colluding group
who can deviate and (weakly) improve the allocation which contradicts the peak GSP property of the rule. \qed

\begin{coro}
\label{cor1}
The edge-fair rule is group strategyproof.
\end{coro}
{\bf Proof.}
It is clear that the edge-fair rule always picks an allocation from $PO^*$.
By Theorem~\ref{}, the result follows if we show that the rule satisfies 
strong invariance.
Following the arguments in the proof of that Theorem, it is enough to consider
an agent $i \in S_{-}$ such that $x_{i} < s_{i}$ and $s'_{i} \geq x_{i}$. 
In this case, the decomposition and maximum flow does not change, and the
overall allocation does not change either: all the breakpoints in the edge-fair
algorithm for the two cases are identical.
\qed

Note that the Egalitarian mechanism is group strategyproof for the same reason. 
This characterization identifies a class of peak group strategyproof mechanisms. 

\subsection{Ranking}

One notion of fairness is that suppose two agents with different peaks have identical connections, then the agent with 
higher peak should have higher net allocation. This is true for the uniform rule where there is only 1 type of divisible good.
This can be formalized in the following way for a general bipartite graph discussed here: ( A similar statement can be made about the 
demanders)
\begin{enumerate}
 \item \textbf{Ranking} (RK) : $s_{i} \leq s_{j} \implies x_{i} \leq x_{j} \forall \hspace{2mm}  i,j \hspace{2mm} \text{such that} \hspace{2mm} f(i) = f(j)$ 
 \item \textbf{Ranking*} (RK*): $s_{i}  \leq s_{j}  \implies s_{i} - x_{i} \leq s_{j} - x_{j} \hspace{2mm} \forall i,j \hspace{2mm} \text{such that} \hspace{2mm} f(i) = f(j)$ 
\item more edges more preference
\end{enumerate}

We start with a proof of statement (i). Suppose $x_{i} > x_{j}$, we show a transfer from agent $i$ to agent $j$ is possible and 
contradicts the lexicographic solution on the edges. 
Construct a new solution $x'$ such that $z'_{kl} = z_{kl} \hspace{2mm} \forall \hspace{2mm}  k \in S \backslash \{i,j\}, \hspace{2mm} l \in f(k), \hspace{2mm}  
z'_{il} = z'_{jl} = \frac{z_{il} + z_{jl}}{2} \hspace{2mm} \forall \hspace{2mm} l \in f(i)$. The allocation $x'$ is clearly feasible and $x$ does not lexicographic dominate $x'$. Hence, we arrive at a contradiction. Using the similar idea of routing the flows from agent $i$ to agent $j$ and by contradiction we can prove statement (ii).

\subsection{Extensions of Uniform Rule}
Both Egalitarian and Edge fair are extensions of the uniform rule whereas edge fair
is a consistent extension of the uniform rule.

%% file: model2.tex
\section{Model 2: Agents on Edges}
\label{s:model2}

As in Section~\ref{s:model1},
we consider the problem of transferring a single commodity from
the set $S$ of suppliers to the set $D$ of demanders using
a set $E$ of edges: each edge $e = (i,j)$ links a distinct
supplier-demanded pair.
However, here we think of the supplier and demander
nodes as passive, whereas each edge $e$ is controlled by a distinct
agent who has single-peaked preferences $R_e$ over the amount of flow
on edge $e$. We think of the ``peak'' $u_e$ of his preference
relation as the capacity of the associated edge.
We write $\mathcal{%
R}$ for the set of single peaked preferences over $\mathbb{R}_{+}$, and $%
\mathcal{R}^{E}$ for the set of preference profiles.
Transfer of the commodity is allowed between
supplier $i$ and demander $j$ only if $(i,j) \in E$.
We let $G = (S \cup D, E)$ be the natural bipartite graph.

As before we focus our attention on 
{\em peak} only mechanisms: in a such a mechanism, the flow depends on
the preferences of the agents only through their peaks, so we could
simply ask each agent $e$ to report their peak $u_e$. We assume that
the supplies $s_i$ and demands $d_j$ are fixed, and the only varying
quantity are the reported peaks (equivalently, edge-capacities).

\paragraph{Pareto flows.}
The set of Pareto efficient
allocations can be complicated because of the peaks of the edge-agents.
For example, suppose there are two suppliers $\{a,b\}$, two demanders
$\{c,d\}$, and edges $\{(a,c), (a,d), (b,d)\}$. Suppose all peaks are
1. Then the flow given by sending 1 unit of flow along the edge $(a,d)$
is Pareto optimal; as is the flow given by sending a unit along each 
of the edges $(a,c)$ and $(b,d)$. In the latter flow 2 units are sent
from the supply to demand nodes whereas only 1 unit is transferred
in the former.


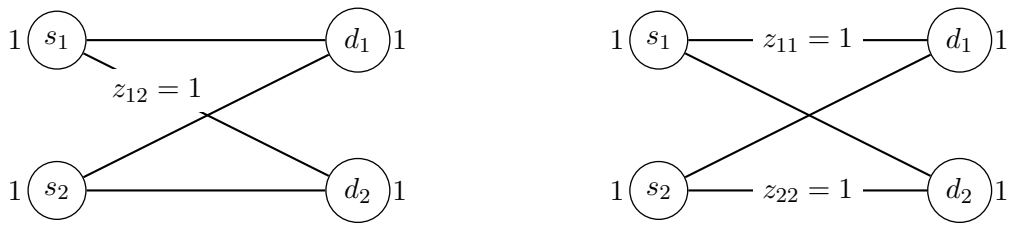
\begin{figure}
 \begin{center}
  \begin{tikzpicture}
   \GraphInit[vstyle=Normal]



	\Vertex[LabelOut,Lpos=180, x=0 ,y=0]{1}
 	\Vertex[x=0 ,y=0,Math,L=s_2]{s2}

 	\Vertex[LabelOut,Lpos=180, x=0 ,y=2]{1}
	\Vertex[x=0 ,y=2,Math,L=s_1]{s1}


	\Vertex[LabelOut,Lpos=0, x=4,y=0]{1}
 	\Vertex[x=4 ,y=0,Math,L=d_2]{d2}

	\Vertex[LabelOut,Lpos=0, x=4 ,y=2]{1}
 	\Vertex[x=4 ,y=2,Math,L=d_1]{d1}

     \Edges(s1,d1)
     \Edges[label={$z_{12}=1$}, style={pos=0.3}](s1,d2)
     \Edges(s2,d2)
     \Edges(s2,d1)



	\Vertex[LabelOut,Lpos=180, x=8 ,y=0]{1}
 	\Vertex[x=8 ,y=0,Math,L=s_2]{s2}

 	\Vertex[LabelOut,Lpos=180, x=8 ,y=2]{1}
	\Vertex[x=8 ,y=2,Math,L=s_1]{s1}


	\Vertex[LabelOut,Lpos=0, x=12,y=0]{1}
 	\Vertex[x=12 ,y=0,Math,L=d_2]{d2}

	\Vertex[LabelOut,Lpos=0, x=12,y=2]{1}
 	\Vertex[x=12,y=2,Math,L=d_1]{d1}

     \Edges[label = {$z_{11}=1$}](s1,d1)
     \Edges(s1,d2)
     \Edges[label = {$z_{22}=1$}](s2,d2)
     \Edges(s2,d1)

	\end{tikzpicture}
   \caption{Every Pareto flow does not allocate a maximum flow} 
   \label{fig2:model2_pareto}
  \end{center}
\end{figure}


In contrast to model 1, therefore, it is possible
that a Pareto optimal flow does not result in a maximum-flow from supply to demand nodes.
For that reason, we assume that the planner implements a max-flow in the given problem
$(G,s,d,u)$, and we consider the question of how this max-flow is
distributed across the edge-agents. In other words, we
focus on the fair division of a max-flow, interpreting max-flow
as a design constraint. Let $\mathcal{F}$ be the set of max-flows.

Restricting ourselves only to max-flows, it is easy to see that the Pareto
set is convex: the average of any two max-flows is itself a max-flow. In
contrast to model 1, any change in flow along an edge affects the agent's
utility directly; in model 1, because the agents were located at the nodes,
it is possible for different edge-flows to give the same allocation to the set
of agents. This implies that every element of $\mathcal{F}$ is a Pareto
allocation.

It is natural to try to formulate this ``edge''-flow problem as a bipartite
rationing problem on an auxiliary graph. For example, consider the Gallai-Edmonds
decomposition for the given network $(G,s,d,u)$, and suppose the partitions are
$S_{+}, S_{-}$ for the suppliers, and $D_{+}, D_{-}$ for the demanders. From
the GE decomposition, every edge between $S_{-}$ and $D_{-}$ carries flow equal
to capacity, so their allocation if fixed in all solutions in $\mathcal{F}$;
likewise for all edges between $S_{+}$ and $D_{+}$. This suggests the following
idea: create a bipartite graph with one node on the left for each edge, and
one node on the right for each element of $S_{+} \cup D_{+}$; each edge that
still remains is incident to either $S_{+}$ or $D_{+}$, but not both; moreover,
the given problem is a {\em rationing} problem in the sense that the nodes
on the right must be fully allocated. Thus it appears that we have rewritten
the flow problem as a bipartite rationing problem of the sort considered in
Section~\ref{s:model1}. That this analogy must be wrong is implied
by the following result.

\begin{prop}
There is no Lorenz Dominant allocation among the edge flows in the set $\mathcal{F}$
\end{prop}

{\bf Proof.} Consider the network of Figure~\ref{fig:model2_LD}. The actual network
is shown in Figure~(a) and the lexicographic solution is shown in~(b).
However, the solution 
$\phi_{c} := \{ z_{11} = 1.4, z_{12} = 1.6, z_{21} = 3, z_{22} = 3.1\}$
is also a maximum-flow; the lex-solution does not dominate this flow, nor
is it dominated by this one. 
\qed


\begin{figure}
 \begin{center}
  \begin{tikzpicture}
   \GraphInit[vstyle=Normal]



	\Vertex[LabelOut,Lpos=180, x=0 ,y=0]{6.1}
 	\Vertex[x=0 ,y=0,Math,L=s_2]{s2}

 	\Vertex[LabelOut,Lpos=180, x=0 ,y=2]{3}
	\Vertex[x=0 ,y=2,Math,L=s_1]{s1}


	\Vertex[LabelOut,Lpos=0, x=4,y=0]{6.6}
 	\Vertex[x=4 ,y=0,Math,L=d_2]{d2}

	\Vertex[LabelOut,Lpos=0, x=4 ,y=2]{4.4}
 	\Vertex[x=4 ,y=2,Math,L=d_1]{d1}

     \Edges(s1,d1)
     \Edges(s1,d2)
     \Edges(s2,d2)
     \Edges(s2,d1)



	\Vertex[LabelOut,Lpos=180, x=8 ,y=0]{6.1}
 	\Vertex[x=8 ,y=0,Math,L=s_2]{s2}

 	\Vertex[LabelOut,Lpos=180, x=8 ,y=2]{3}
	\Vertex[x=8 ,y=2,Math,L=s_1]{s1}


	\Vertex[LabelOut,Lpos=0, x=12,y=0]{6.6}
 	\Vertex[x=12 ,y=0,Math,L=d_2]{d2}

	\Vertex[LabelOut,Lpos=0, x=12,y=2]{4.4}
 	\Vertex[x=12,y=2,Math,L=d_1]{d1}

     \Edges[label = {$z_{11}=1.5$}](s1,d1)
     \Edges[label = {$z_{12}=1.5$},style={pos=0.3}](s1,d2)
     \Edges[label = {$z_{22}=3.2$}](s2,d2)
     \Edges[label = {$z_{21}=2.9$}, style={pos=0.3}](s2,d1)

	\end{tikzpicture}
   \caption{Absence of Lorenz Dominance element in Model 2} 
   \label{fig:model2_LD}
  \end{center}
\end{figure}
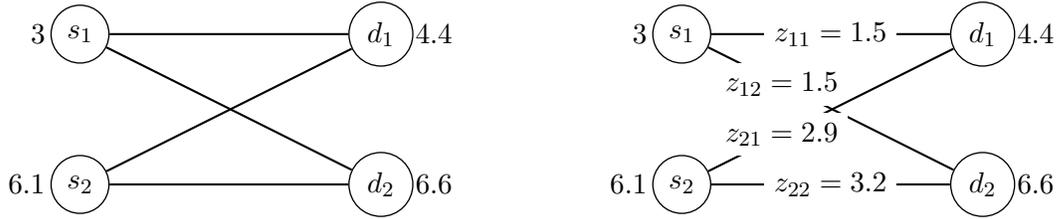


\bigskip

{\bf Remark.}
If we draw the bipartite graph suggested in the discussion before the
statement of the proposition, and treat it as a bipartite rationing
problem, we find that edges $(1,1)$ and $(2,1)$
will carry a flow of 1.5 and 3.05 units each, and this exceeds the
total demand at $D1$. These implied ``side-constraints'' are not accounted
for in translating the given problem to a bipartite rationing problem.

\paragraph{Allocation Rules.}
We can apply the edge-fair rule discussed earlier on this model as well.
The edge-fair rule finds a lex-optimal max-flow. It is clear that the
rule is also edge consistent. Our next result shows that every
edge-consistent rule is also group strategyproof.

\begin{theorem}
Fix a graph $G$ with the supply vector $s$ and demand vector $d$. Suppose
we have an allocation rule that maps reports of edge-capacities to a flow.
Every edge-consistent allocation rule is group strategyproof.
\end{theorem}

\noindent
{\bf Proof.} 
Consider a coalition of agents $A = \{e \in E | u_{e}' \neq u_{e}\}$, i.e., they misreport 
their true peaks.
Let the misreported profile be denoted by 
$R' \in \mathcal{R}^{|E|}$ and the resulting network by $G'$. Note that the edge-fair rule 
always results in an allocation $z \leq u$, hence any
agent $e \in A$ should report $u'_{e} \geq z_{e}$; 
otherwise, $z'_{e} \leq u'_{e} < z_{e}$ and the agent 
$e$ is worse off in profile $R'$.
Let $B: = \{e \in A | z_{e} = u_e \}$. The agents in $B$ 
should have the allocation $z_{e}' = u_{e}$ when the reports are $R'$
as every such agent received their peak allocation in 
profile $R$. 
Consider the graph $ G \setminus B$ by removing the agents in
$B$ to form the reduced graph 
$(G \setminus B, s_{G \setminus B},d_{G \setminus B},u')$,
where $s_{G \setminus B},d_{G \setminus B}$ are the adjusted peaks
of supply and demand nodes respectively after fixing the flow 
on the agents in $B$.

By edge-consistency of the rule, the allocation
$z'_{e} = z'_{e}(-B)$ for all $e \in G \setminus B$.

From the discussion above, the report of every 
agent $e \in  G \backslash B$
is such that $c'_{ij} \geq z_{ij}$. Also, note that $z_{ij} < u_{ij} \hspace{2mm} \forall \hspace{2mm} ij \in  G \backslash B$. 

By increasing the capacity of an unsaturated edge, the total value of the maximum flow does not change and the bottleneck points remain the same when the edge fair rule is 
applied to components $(G \backslash B, s_{G \backslash B},d_{G \backslash B},c')$ and $(G \backslash B, s_{G \backslash B},d_{G \backslash B},c)$. Hence,
every agent $ij \in A$ receives the allocation $z'_{ij} = z_{ij}$. 
\qed

\bigskip

We next turn to equity properties of allocations and allocation rules. 
Given that different edges may connect possibly different suppliers and
demanders who may have supply or demand different amounts of the commodity,
one has to be careful in formulating these notions. Following
Bochet et al.~\cite{bims}, we formulate these properties for a pair of
agents (equivalently, edges). In general these properties take the
following form: Fix a problem $(G,s,d,u)$, and consider the allocation 
$z$ given by a rule $\varphi$. For every pair of edges $e$ and $e'$, 
fix the flows on all edges other than $e$ and $e'$ and ask if there is 
a ``better'' feasible flow in $\mathcal{F}$.

An allocation is envy-free if whenever $e$ prefers $z_{e'}$ to $z_{e}$ (for some
agents $e$ and $e'$), there is no other allocation 
$\hat{z} \in 
\{ z' \in \mathcal{F} \mid z'_{f} = z_{f}, \;\; \forall f \not = e, e'
\;\; \}$
such that $e$ prefers $\hat{z}$ to $z$. An allocation $z$ satisfies equal treatment
of equals (ETE) if for each $e$ and $e'$ with $u_e = u_{e'}$, there is no
other allocation 
$\hat{z} \in 
\{ z' \in \mathcal{F} \mid z'_{f} = z_{f}, \;\; \forall f \not = e, e'
\;\; \}$
with $|\hat{z}_{e}-\hat{z}_{e'}| < |z_e - z_{e'}|$.

The following result shows the relationship between these two properties.
\begin{prop} 
Consider the problem $(G,s,d,u)$ and an allocation rule $z$ that makes a
selection from the Pareto set $\mathcal{F}$. If $z$ is envy-free it satisfies
ETE.
 \end{prop}

\noindent
{\bf Proof.}
We mimic the proof in Bochet et al. ~\cite{bims} here. Suppose the rule $z$ violates ETE, we would like 
to show it violates No Envy or the flow is $\notin$ PO$^{*}$. Fix a profile $R^{E}$ and two edge agents $e$ and $e'$ such that
$u_{e}[R_{e}] = u_{e'}[R_{e'}] = c^{*}$ and suppose there exists $z'$ satisfying the definition above . Now, we have that, $z'_{e} + z'_{e'}
= z_{e} + z_{e'}$ because $z$ and $z'$ coincide on $E \backslash \{e, e'\}$. Assume without loss $z_{e}(R) < z_{e'}(R)$, then only two 
cases are possible: $z_{e}(R) < z'_{e} \leq z'_{e'} < z_{e'}(R)$ or $z_{e}(R) < z'_{e'} \leq z_{e} < z_{e'}(R)$. \\
Assume first case: $c^{*} \geq z_{e'}(R)$ implies a violation of No Envy. Now in case (ii), the allocation $z''_{e}  = \frac{z_{e}+z'_{e}}{2} 
\hspace{2mm} \forall ij \in E$ is such that $z''_{e} \in PO^{*}$ and we are in case (i) again.
\qed

By construction, the edge-fair rule selects a maximum flow allocation from the Pareto set. The edge-fair rule also 
finds an envy-free allocation. 
Define the set of agents $A:=\{e| e \in E, z_{e} > 0 \}$, $B:=\{e| e \in E, z_{e} = 0 \}$, $E = A \cup B$. If $z_{e} > 0 $
under the edge fair rule, then the agent $e$ carries a \emph{positive} flow in some maximum flow solution. Similarly, $e \in B$ 
do not carry a positive flow in any maximum flow solution. So even if $z_{e}R_{e'}z_{e'}$ for some $e' 
\in B, e \notin B$, there is no maximum flow solution in which $z_{e'} > 0$ to possibly redistribute and improve the allocation of agent $e'$.
On the other hand, $e',e \in A$ implies $e$ is in a higher bottleneck set than $e'$ since the allocation rule is monotone. Suppose,
there is envy through the solution $z'$, consider the solution $\frac{z' + z}{2}$, which is still feasible because the set $\mathcal{F}$ is convex.
This is a contradiction to the earlier obtained solution of the LP at the step when $e'$ was a bottleneck. Hence, edge fair satisfies no envy in this model
thus treats equals equally.